\documentclass[11pt]{amsart}
\usepackage{amssymb,amsmath}
\usepackage[sc]{mathpazo}

\usepackage[colorlinks=true]{hyperref}
\hypersetup{urlcolor=blue, citecolor=blue, linkcolor=blue}

\usepackage{xcolor}
\usepackage{enumitem}

\theoremstyle{plain}
\newtheorem{mainthm}{Theorem}
\newtheorem{mainclly}{Corollary}
\newtheorem{theorem}{Theorem}[section]
\newtheorem{thm}[theorem]{Theorem}

\newtheorem{proposition}[theorem]{Proposition}

\newtheorem{lemma}[theorem]{Lemma}

\theoremstyle{definition}
\newtheorem{remark}[theorem]{Remark}
\newtheorem{example}[theorem]{Example}
\newtheorem{definition}[theorem]{Definition}

\newcommand{\F}{\mathcal F}

\newcommand{\field}[1]{\mathbb{#1}}
\newcommand{\real}{\field{R}}

\newcommand{\integer}{\field{Z}}

\renewcommand{\epsilon}{\varepsilon}
\newcommand{\eps}{\varepsilon}

\title[Periodic points for expansive pseudo-groups]{On the number of periodic points for expansive pseudo-groups}

\author{Pablo D. Carrasco}
\address{ICEx-UFMG, Avda. Presidente Antonio Carlos 6627, Belo Horizonte-MG, BR31270-90; Department of Mathematics, Southern University of Science and Technology, No 1088, Xueyuan Rd., Xili, Nanshan District, Shenzhen, Guangdong 518055, China.}
\email{pdcarrasco@gmail.com}
\author{Elias Rego}
\address{Department of Mathematics, Southern University of Science and Technology, No 1088, Xueyuan Rd., Xili, Nanshan District, Shenzhen, Guangdong 518055, China.}
\email{rego@sustech.edu.cn}
\author{Jana Rodriguez-Hertz}
\address{SUSTech International Center for Mathematics and Department of Mathematics, Southern University of Science and Technology, No 1088, Xueyuan Rd., Xili, Nanshan District, Shenzhen, Guangdong 518055, China.}
\email{rhertz@sustech.edu.cn}

\thanks{PDC was partially supported by CAPES-PRINT 88887.716916/2022-00
ER was partially supported by NSFC 12250710130
JRH was partially supported by NSFC 12161141002 and 12250710130}



\begin{document}

	\begin{abstract}

	In this work we consider foliations of compact manifolds whose holonomy pseudo-group is expansive, and analyze their number of compact leaves. Our main result is that in the codimension-one case this number is at most finite, and we give examples of such foliations having  one compact leaf.
	
\end{abstract}

\maketitle

	\section{Introduction}
Let $M$ be a compact manifold and $\F$ a foliation on it. In this work we will be concerned with understanding possible restrictions for the number of compact leaves that $\F$ may have, provided that its holonomy pseudo-group has some rich dynamics. Concretely, we will consider $\F$ having an expansive pseudo-group: in this case $\F$ is said to be an \emph{expansive foliation}.

This type of question is of course very natural, and has its roots in the analogous problem for group actions. For the classical case when one has a (say, smooth) action of $\real$ or $\integer$ on $M$ the answer is known: any such action has countably many closed orbits (periodic points). We note here that in the rank-one case there are no expansive actions of codimension-one.

It turns out that there exist codimension-one expansive foliations. The contribution of this work to the theory consists precisely in answering the aforementioned question in this setting.

\begin{mainthm}\label{finitecompactleaves}
  Let $\F$ be a $\mathcal{C}^{1,0+}$ expansive codimension-one foliation of a compact boundaryless manifold. Then it has at most finitely many closed leaves.
  \end{mainthm}  

Observe also that due to Reeb's stability, if $\F$ is expansive then any closed leaf has to have infinite holonomy.

In Section \ref{examples} we indicate how to use examples given by S. Hurder of expansive actions in order to get codimension-one foliations having closed orbits, and also with exceptional minimal sets. The proof of Theorem \ref{finitecompactleaves} diverges from the arguments used in the classical case, relying instead in a variation of the Epstein hierarchy for compact foliations \cite{Per3Man}.

For completeness, let us mention that in general codimension the number of compact leaves is at most countable. See Theorem \ref{countable}. The orbit foliation of an Anosov flow provides an example of an expansive foliation with infinitely many closed leaves.

We highlight the following direct consequence of our main Theorem. For the definition of expansivity in actions see the next section.

\begin{mainclly}\label{actinsorbit} Let $\varPhi$ be a $\mathcal{C}^1$ locally-free action of a Lie group $G$ acting on a compact boundaryless manifold $M$, with $\dim M-\dim G=1$. If $\varPhi$ is expansive then it has at most finitely many compact orbits.
\end{mainclly}

\paragraph{\textbf{Expansive actions on manifolds and foliations}}Expansivity has always played an important role in classical dynamical systems. As a weak form of sensitivy to initial conditions, it was noted as a useful property of hyperbolic systems, such as geodesic flows in negativively curved compact manifolds \cite{AnosovThesis}, which led to its isolation for study. See the treatises \cite{Hasselblatt2003} and \cite{Denker1976} for further discussion.

To give some significative examples in the case of $\integer^d, d\geq 1$  actions we mention the classical work of Ruelle \cite{thermodynform}, who made extensive use of expansivity for studying interactions in lattices, and of Katok-Schmidt, who initiated the study of the rigidity properties of such actions. Expansive actions of more general groups acting on the circle were considered in the precursor work of Hurder \cite{Hurder2000}. More recent contribution for general groups can be found in the work of Bonomo et al. \cite{Bonomo2017}, and of Arbieto-Rego \cite{AR}.

As for pseudo-groups, foliations with expansive holonomy were introduced by  Inaba and Tsuchiya in \cite{IT}, work that motivated what it is done here. We would like to point out that we have purposely kept our bibliography to a minimum in order to simplify the reading, and only mentioned some representative examples. The reader will be more served by looking at the references in the more recent works cited above.

\smallskip 
 
\paragraph{\textbf{Organization}}
In section \ref{preliminaires} we present precise definitions for the concepts used throughout the text and state some known results. In section \ref{compact} we prove our main results, starting with Theorem \ref{countable}, which is simpler, and then building to \ref{finitecompactleaves}, in particular with a discussion of Epstein's hierarchy. Lastly, in section \ref{examples} we explain how to use some examples of \cite{Hurder2000} to show optimality of our results.

\section{Preliminaries}\label{preliminaires}	
	In this section we establish notation and collect some basic definitions and results. To keep the discussion succinct, it is assumed that the reader is acquainted to basic foliation theory as discussed for example in \cite{FoliationsI}.

    Unless otherwise stated, $M$ denotes a closed (compact, boundaryless) Riemmanian manifold of dimension $m$. Its induced distance is denoted by $d_M$.
    
\newpage

    \subsection{Basic concepts}

\smallskip

\paragraph{\textbf{$\mathcal{C}^{1,+0}$-foliation}} By a $\mathcal{C}^{1,+0}$-foliation on $M$ we mean a partition $\mathcal{F}=\{L_x\}_{x\in M}$ into $p$-dimensional submanifolds of class $\mathcal{C}^1$ (called the leaves of the foliation), and such that 
 \[
    T\mathcal{F} :=\bigsqcup_{x\in M} L_x 
 \]
 defines a continuous sub-bundle of $TM$. If $T\mathcal{F}$ is of class $\mathcal{C}^{r-1}, r\geq 1$ we say that $\F$ is of class $\mathcal{C}^r$.

 It follows that each point $x\in M$ has an open neighborhood $U$ and a homeomorphism $\psi_U:(-1,1)^p\times (-1,1)^{m-p}\to U$ such that for every $w\in (-1,1)^{m-p}$ the map $\psi_U(\cdot,w):(-1,1)^p\times\{w\}\to U$ is a $\mathcal C^1$ ($\mathcal C^r$ if $\F$ is $\mathcal C^r$) diffeomorphism onto a connected component of the intersection of some leaf with $U$. In this case the image (which is a $p$-dimensional embedded disc) is called a plaque of the foliation. The set $U$ is a foliation box of $\F$, and a covering of $M$ by foliation boxes is called a foliation atlas.

It is said that $\F$ has dimension $p$ and codimension $q=m-p$.

\smallskip 

\paragraph{\textbf{Convention for the rest of the article}} Unless otherwise stated, from now on $\F$ denotes a $\mathcal{C}^{1,0+}$ foliation.

\begin{example} Consider $G$ a Lie group acting (smoothly) on a compact manifold $M$, and suppose that the action is locally free, that is, for  each $x\in$ the stabilizer $G_x$ is discrete. It is an almost direct consequence of Frobenius theorem that $\{G\cdot x\}_{x\in M}$ defines a foliation on $M$. This is referred to as the orbit foliation induced by (the action of) $G$.
\end{example}

\smallskip

\paragraph{\textbf{Holonomy}} Let $\eta_0>0$ be the injectivity radius of the exponential structure associated to the metric of $M$. Given $0<\eta\leq\eta_0$ and $x\in M$ we write
\[
	T(x,\eta)=\{\exp_x(v):v\in T_xM, v\perp TL_x, \|v\|<\eta\}.
\] 
By reducing $\eta_0$ if necessary it follows that $T(x,\eta)$ is transverse to any plaque of $\F$: we refer to the set $T(x,\eta)$ as a local transversal through $x$, and write $T(x)=T(x,\eta_0)$.

A continuous curve $\gamma:[a,b]\to M$ is a $\F$-curve if its image is contained if some leaf of $\F$. Given such a curve write $x=\gamma(a), y=\gamma(b)$ consider a finite partition $t_0=a<t_1<\ldots<t_r=b$ so that $\gamma_i=\gamma|([t_i,t_{i+1}]$ is contained in a foliated box of $\F$. By lifting $\gamma_i$ to nearby plaques one defines an embedding $h_{\gamma}:T(x,\eta)\to T(y)$, where $\eta$ is sufficiently small. The map $h_{\gamma}$ is the holonomy transport defined by $\gamma$, and it is well known that it only depends on the homotopy class of $\gamma$, provided that the $x,y$ are maintained fixed. See \cite{FoliationsI}.

\begin{definition}
The holonomy pseudo-group of $\F$ is the set $\mathcal{H}(\F)$ that consists of all possible $h_{\gamma}$ constructed as above.

The holonomy group of the leaf $L$ at the point $x$ is 
\[
	\mathcal{H}_x(L):=\{\text{germ}_x(h):h\in \mathcal{H}(\F), h(x)=x\}.
\]
\end{definition}
It follows that $\mathcal{H}_x(L)$ is a group and there is a natural group epimorphism $\Pi_{L,x}:\pi_1(L,x)\to \mathcal{H}_x(L)$. Since for $x,y$ in the same leaf the groups $\mathcal{H}_x(L), \mathcal{H}_y(L)$ are isomorphic we omit the dependence on the base point, whenever is convenient.

\paragraph{\textbf{Minimal and saturated sets}} We say that $A\subset M$ is \textit{saturated} if it is a union of leaves of $\F$. If $A$ is saturated and $L\subset A$ is a leaf one can consider the restricted holonomy group of $L$ in $A$, $\mathcal{H}(L)|A=\{h|A:h\in \mathcal{H}(L)\}$.

  \begin{definition}
  Let $A\subset M$.
  \begin{enumerate}
      \item $A$ is a \textit{minimal set} if it its compact, saturated and $A$ does not contain proper compact saturated subsets.
      \item $A$ is an \textit{exceptional minimal set} if it is minimal and $A\neq M$ and $A$ is not a compact leaf.

  \end{enumerate}
  The foliation $\F$ is minimal if $M$ is a minimal set.
  \end{definition}
By compactness of $M$, necessarily $\F$ has minimal sets.

\smallskip

\paragraph{\textbf{Expansive foliations}} The foliation $\F$ is expansive if there exists some $\eps>0$ such that if $x\in M$, $y\in T(x,\eps)\setminus x$ then there exists some holonomy map $h:T(x,\eta)\to T(z)$ with $y\in T(x,\eta)$ and $d_M(h(x),h(y))>\eps$. The number $\eps$ is said to be an expansive constant for $\F$. 

The original definition of expansive foliation given by Inaba and Tsuchiya is presented in a different but equivalent way. See \cite{Wa} for the pertinent discussion.

For codimension-one foliations one has the following characterization 

\begin{theorem}[\cite{IT}]\label{Locmin}
A codimension-one foliation $\F$ is expansive if, and only if, there is a finite family of open saturated sets $U_1,...,U_n$ such that
\begin{enumerate}
    \item  If $L\subset U_i$, then $L$  is dense in $U_i$, for $i=1,...,n$.
    \item The union $\bigcup_{i=1}^n U_i$ is dense in $M$.     
\end{enumerate}
\end{theorem}

\smallskip 
 
\paragraph{\textbf{Orientability of expansive foliations}} The foliation $\F$ is said to be
\begin{itemize}
	\item \emph{orientable} if $T\F$ is orientable.
	\item \emph{transversely orientable} if the normal bundle of $T\F$ is orientable.
\end{itemize}

\begin{proposition}[\cite{FoliationsI}, Proposition 3.5.1]\label{lift}
There is finite covering space $\pi:\overline{M}\to M$ (of degree at most $4$) and $\overline{F}$ a foliation on $\overline{M}$ such that 
\begin{enumerate}
	\item $\pi$ maps leaves of $\overline{F}$ diffemorphically onto leaves of $\F$.
	\item $\overline{M}$ is orientable, and $\overline{F}$ is both orientable and transversely orientable.
\end{enumerate}
\end{proposition}

\begin{proposition}
  For $\pi:\overline{M}\to M$ as above, if $\F$ is expansive then $\overline{\F}$ is expansive.
\end{proposition}

\begin{remark}
The same holds whenever $\pi:\overline{M}\to M$ is a finite covering space, and $\overline{\F}$ is the pull-back foliation of $\F$ by $\pi$. We will not use this in this article. 
\end{remark}

\begin{proof}
 Lift the Riemannian metric of $M$ via $\pi$. Let $\eps>0$ be an expansive constant for $\F$ and consider a covering $\mathcal{U}$ of $M$ by foliation boxes, with diameter so small that for each $U\in\mathcal{U}$, $\pi$ maps isometrically each connected component of $\pi^{-1}(U)$ to U. The set $\overline{\mathcal{U}}=\{\text{connected components of }\pi^{-1}(U):U\in\mathcal{U}\}$ is a covering by foliation boxes for $\overline{\F}$, and is clear that $\pi$ conjugates isometrically each element of $\mathcal{H}(\F)$ to an element of $\mathcal{H}(\overline{\F})$ (by lifting $\F$-curves). From here it follows that $\epsilon$ is an expansive constant for $\overline{\F}$.
\end{proof}

\smallskip 
 
\paragraph{\textbf{Convention for the rest of the article}} Due to the above we will consider from now on that both $\F$ and the manifold are oriented.

\smallskip 
 
\paragraph{\textbf{Expansive actions}} Let $\varPhi:G\times M\to M$ be a (differentiable) action of a group $G$ which is either discrete or a connected Lie group. The action $\varPhi$ is \emph{expansive} if

\begin{itemize}[leftmargin=*]
	\item \textbf{Discrete case:} there is $\eps>0$ such that if $x,y\in M$ and $x\neq y$, then there is $g\in $ such that $d_M(gx,gy)>\eps$. 
	\item \textbf{Continuous case  \cite{AR}}: for any $\eps>0$ there is $\delta>0$ satisfying the following. Given $x,y\in M$ and $h\in\mathcal C^0{G,G}$ with $h(e)=e, \sup_{g\in G} d_M(h(g)\cdot x,g\cdot y)<\delta$, it holds that $y=g_0\cdot x$ for some $g_0\in G$ with $d_G(g_0,e)<\eps$ 
\end{itemize}

If $\varPhi$ is an expansive locally-free action of a continuous Lie group, then its orbit foliation is expansive. See Theorem 2.15 in \cite{AR}.

\subsection{Structure of the set of compact leaves}Denote

\begin{align*}
\mathcal{K}(\F):=\bigcup\{L:L\in \F, L\text{ compact}\}
\end{align*}

For the proof of Theorem A we will study the holonomy of the induced foliation in the saturated set $\mathcal{K}(\F)$, together with the following.

\begin{theorem}[Haefliger \cite{VarFeu}]\label{Haefliger}
If $\F$ is of codimension one then $\mathcal{K}(\F)$ is compact.
\end{theorem}

We also recall the classical stability Theorem due to Reeb.

\begin{theorem}[Reeb Local Stability Theorem]\label{Reeb}
Let $E \subset M$ be a locally compact saturated set and suppose that $L \subset E$ is compact and satisfies $\# \mathcal{H}(L)|E< \infty$. Then $L_x$ has arbitrarily small saturated neighborhoods that consists of compact leaves, and each one of these leaves $L'$ satisfies $\# \mathcal{H}(L')|E<\infty$. 
\end{theorem}

The version above can be found for example in Appendix A of \cite{compactfoliationsart}.

    \section{Compact Leaves of Expansive Foliations}\label{compact}

This section is dedicated to the proof of the our main result. Before doing that we will discuss a more basic result stating that in general codimension the number of compact leaves is at most countable.

We recall that we are assuming that $\F$ is a $\mathcal{C}^{1,+0}$ expansive foliation on the closed manifold $M$, which is both orientable and transversally orientable (therefore $M$ is orientable). We fix once and for all an expansivity constant $\epsilon$ of $\F$.

\subsection{Countably many compact leaves}
 
Let us first note:

\begin{thm}\label{countable}
Let $\F$ be a $\mathcal{C}^{1,0+}$ expansive foliation of a compact boundaryless manifold. Then it has at most countably many closed leaves.
\end{thm}

\begin{proof}
For a subset $A\subset L\in\F$ denote by $\mathrm{Vol}(A)$ its volume with respect to the induced Riemannian metric on $L$ by $M$. Write
\[
	\mathcal{K}(\F)=\bigcup_{N\geq 1} \{L\in\F:\mathrm{Vol}(L)\leq N\}=\bigcup_{N\geq 1} \mathcal{K}_N.
\]
It is enough to show that each $\mathcal{K}_N$ consist of finitely many leaves. Assume otherwise: consider a sequence $(L_n)_n \subset \mathcal{K}_N$ of pairwise different leaves converging in the Hausdorff topology to some set $A$. Then $A$ is compact, connected and saturated. 

 Cover $A$ with a finite family $U_1,...,U_k$ of foliated boxes with diameter smaller than $\eps$. Denote $$U=\bigcup_{i=1}^k U_i.$$ Since $L_n$ converges to $A$, for $n$ big enough, we have $L_n\subset U$. Notice that we can choose the neighborhoods $U_i$ satisfying the following property: There are $v_-,v_+>0$ such that $v_-\leq \mathrm{Vol}(P)\leq v_+$, for every plaque $P\subset U_i$ and for every $i=1,..,k$. 

 For $n$ large the leaf $L_n$ can intersect each $U_i$ in at most $\frac{N}{k v_-}$ plaques, therefore it follows that $A$ has finite volume, and is thus a compact leaf. This is a contradiction, since $L_n$ is uniformly $\eps$-close to $A$ for $n$ large, therefore no holonomy map can separate $A$ and $L_n$, which contradicts the definition of expansivity.
\end{proof}

\subsection{Codimension-one expansive foliations - Proof of Theorem \ref{finitecompactleaves}}	 

In what follows we assume additionally that $\F$ has codimension one. Our goal is to show that $\mathcal{K}(\F)$ contains at most finitely many leaves, and for this we ellabore on its structure. We consider the filtration of sets $\{B_{\alpha}\}_{\alpha}$ indexed by the ordinals, where

\[
     B_{\alpha}=\begin{cases}
     B_0=\mathcal{K}(\F)\\
     B_{\alpha}=\{x\in B_{\alpha-1}; \#\mathcal{H}(L_x)|B_{\alpha-1}=\infty\},  \text{ if }\alpha\text{ is a successor ordinal }\\
     B_{\alpha}=\bigcap_{\beta<\alpha} B_{\beta}, \text{ otherwise}
     \end{cases}
 \]
 This is a variation of the so called Epstein filtration of Bad Sets, introduced in \cite{Per3Man} by D.B.A Epstein in the context of $3$-dimensional flows whose leaves are all compact, more  See also \cite{BadSets2}.

\begin{lemma}
Each $B_{\alpha}$ is a compact saturated set, and there exists a successor ordinal $\alpha_0$ such that $B_{\alpha_0-1}\neq \emptyset$ and $B_{\alpha_0}=\emptyset$. 
\end{lemma}

\begin{proof}
First, $B_0$ is compact due to Theorem \ref{Haefliger}. The fact that each $B_{\alpha}$ is compact in $B_0$ is consequence of theorem \ref{Reeb}. Clearly every $B_{\alpha}$ is saturated.

Let $\Omega$ be the first uncountable ordinal and suppose by means of contradiction that $B_{\Omega}\neq\emptyset$. For each $\alpha<\Omega$ choose $x_{\alpha}\in B_{\alpha}\setminus B_{\alpha+1}$, and consider $x_{\alpha_0}\in \{x_\alpha\}'$ (which exists since $\{x_\alpha\}$ is uncountable). Reeb stability theorem now gives a contradiction, since points $x_{\alpha}$ for $\alpha>\alpha_0$ cannot enter any foliated neighborhood of $x_{\alpha_0}$ in $B_{\alpha_0}\setminus B_{\alpha_0+1}$.
\end{proof}

Now we finish the proof of Theorem \ref{finitecompactleaves}, which essentially amounts of showing that $\alpha_0=0$.

Suppose that there exists infinitely many compact leaves $\{L_n\}_{n\geq 1} \subset  B_0$ with $L_n \in B_{\alpha_n}$, and consider an accumulation point $L$ in the Hausdorff topology of this sequence. Since $B_0$ is compact it follows that $L$ is a compact leaf, therefore $L\in B_{\alpha_L}$ for some ordinal $\alpha_L$. By Reeb's stability we can find $U$ open saturated neighborhood of $L$ in $B_{\alpha_L}\setminus B_{\alpha_L+1}$. It then follows that for $n$ large $L_n \subset U$. But $U$ can be taken so that $U\subset \bigcup_{x\in L} T(x,\frac{\eps}{2})$, which means that no holonomy separates $L$ from $L_n$. This contradicts expansivity of $\F$.

\section{Some examples of Expansive Foliations}\label{examples}

       This section is devoted to examples of expansive foliations. 

       \paragraph{\textbf{Notation.}} $\F$ foliation, $L\in \F$. For $x\in M, r>0$ we write $L_{r}(x)$ for the open ball in $L$, centered at $x$ and with radius $r$.

      \subsection{Anosov Actions}

      We start with invariant foliations associated to Anosov Actions; see \cite{ErgAnAc} for the undefined terms. Below $G$ denotes a Lie Group and $\varPhi:G\times M\to M$ an Anosov action on the closed manifold $G$ (in particular, locally free). For some regular element $g\in G$ consider the associated invariant foliations $\F^{c}_g,\F^s_g, \F^{u}_g, \F^{cs}_g,\F^{cu}_g$, the center (orbit), stable, unstable, center-stable and center-untable foliatons respectively.

      Hirsch, Pugh and Shub in \cite{HPS} considered a different type of expansiveness associated to these kind of foliations, which they coined ``plaque expansiveness''.

       \begin{definition}
       Let $\F$ be a foliation of a closed manifold $M$ which is invariant by a homeomorphism $f:M\to M$. For $\delta>0$, a sequence $\{x_i\}_{i\in \integer}$ is said to be a \textit{$\delta$-pseudo-orbit for $\F$} if
       \[
       f(x_i)\in L_{\delta}( x_{i+1}), \forall i\in\integer.
       \]
       \end{definition}

     \begin{definition}[plaque-expansiveness]
		 In the same setting as above $\F$ is \textit{plaque-expansive} if there exists $\delta>0$ so that for any two $\delta$-pseudo-orbits $\{x_i\}$ and $\{y_i\}$ for $\F$ satisfying $\sup_{i\in\integer} d_M(x_i,y_i)\leq \delta$, one has $x_0\in L_{y_0}(\delta)$.
	\end{definition}

     In Theorem 7.2 of \cite{HPS} it is proven that $\F^c_g$ is plaque expansive, which in turn implies without too much trouble that $\F^{cs}_g,\F^{cu}_g$ are plaque expansive as well.

     \begin{proposition}\label{plqexpactions}
	  Let $\varPhi$ be a locally-free $G$-action on a closed manifold $M$ and suppose that there is some $g\in G$ such that ${\varPhi_{g}}$ is plaque-expansive. Then the orbit foliation of $\varPhi$ is expansive.
	       
	\end{proposition}
	
	\begin{proof}
		Suppose $f(x)=g_0\cdot x$ is plaque-expansive, and take $0<\delta\leq \eta$ a plaque expansivity constant of $\F$ that satisfies
		\[
           y\in T(x,\delta)\setminus \{x\}\Rightarrow L_{\delta}(x)\cap L_{\delta}(y)=\emptyset.
		\]
          By considering the sequences $\{x_n=g_0^n\cdot x=f^n(x)\}, \{y_n=g_0^n\cdot y=f^{n}y\}$, it follows by plaque-expansivenes that if $y\in T(x,\delta)\setminus \{x\}$ then necessarily $\sup_n d_M(x_n,y_n)>\delta$. This shows on each leaf  $L\in \F$ the subgroup generated by $f$ induces an expansive sub-group of $\mathcal{H}(L)$, which shows that $\F$ is expansive.
	\end{proof}
     
     It follows from the above that $\F^c_g$ is expansive, which in turn implies that $\F^{cs}_g,\F^{cu}_g$ are also expansive. This is consequence of the fact that if an expansive foliation $\F$ sub-foliates a foliation $\mathcal{G}$ (meaning, each leaf of $\mathcal{G}$ is union of leaves of $\F$), then $\mathcal{G}$ is expansive, because in this case there is a natural inclusion $\mathcal{H}(\F) \subset \mathcal{H}(\mathcal G)$.

     The foliations $\F^s_g, \F^u_g$ seem harder to deal with in general, but we can deal with the algebraic case. We consider the following situation, which encapsulates the case for typical algebraic Anosov actions and their correspoding stable and unstable (horocyclic) foliations, see \cite{FirstCoh}.

     Consider $G$ a continuous semi-simple Lie group with finite center, non-trivial compact factor, and let $\Gamma<G$ be a uniform lattice. Write $G=KAN$ its Iwasawa decomposition and let $X=\Gamma/ G$. Then $N$ acts on $X$ by right multiplication, and this action is usually referred as a ``horocyclic flow''.

    \begin{proposition}
    The orbit foliation of $N$ is not expansive. 
    \end{proposition}

    \begin{proof}
       Indeed, expansive foliations have positive geometric entropy (\cite{AR}), which in turn implies for foliated actions that their entropy (as actions) is also positive. See for example the first part of Theorem 3.4.3 in \cite{Walczak2004}.

       It is well known, on the other hand, that the action of any unipotent action (as the one defined by $N$ on $X$) has zero entropy (cf. Theorem 7.6 in \cite{Einsiedler2008DiagonalAO}). 
    \end{proof}


     
    



\subsection{Codimension-one Expansive foliations with compact leaves}

In what follows we explain why the result of our Theorem A is sharp. Let us recall that given $B,F$ closed manifolds and a representation $\Gamma:\Pi_1(B)\to \mathrm{Diff}^r(F)$ one can construct a foliated fiber bundle $\pi: E_{\Gamma}\to B$ with fiber $B$, whose foliation $\F$ is transverse to the fibers; this is called the suspension of $\Gamma$ (with the corresponding data). In addition, the holonomy group of $\mathcal{F}$ is conjugated to $\Gamma$. See \cite{FoliationsI}.

It follows in particular that if $\Gamma$ is discrete and the action of $\mathrm{Im}(\Gamma)$ is expansive on $F$, then $\F$ is an expansive foliation (of codimension $\dim F$).

\begin{example}[Codimension-one expansive foliation with a single compact leaf]\label{excpctleaf}

In Example 8.11 of \cite{Hurder2000} S. Hurder gives an example of an expansive action of a subgroup $\Gamma=\langle f,g\rangle<\mathrm{Diff}^{\infty}_{+}(S^1)$ having a unique fixed point, with minimal action on the complement: here $\mathrm{Diff}^{\infty}_{+}(\mathbb{S}^1)$ denotes the orientation preserving (smooth) diffeomorphisms of the circle. 

Let $\Sigma$ be the bitorus: then $\pi_1(\Sigma)\sim \langle a_1,b_1,a_2,b_2|[a_1,b_1][a_2,b_2]=1\rangle$. It follows that $\langle a_1,a_2\rangle<\pi_1(\Sigma)$ is free. Define the representation $\Gamma:\pi_1(\Sigma)\to \mathrm{Diff}^{\infty}_{+}(\mathbb{S}^1)$ by
\begin{align*}
&\Gamma(a_1)=f, \Gamma(a_2)=g\\
&\Gamma(b_1)=\Gamma(b_2)=Id.
\end{align*}
Then the suspension of $\Gamma$ provides an example of an expansive codimension-one foliation with finite but non-zero compact leaves: indeed, the leaf corresponding to the common fix point $p$ of $f,g$ is compact. This is easily noted since the leaf that contains $p$ is a minimal set of the foliation, and thus compact.

Similarly, S. Hurder constructs an example of an expansive action $\Gamma=\langle f,g\rangle<\mathrm{Diff}^{\infty}_{+}(S^1)$ having a unique minimal set, which is a Cantor set. Suspending this as above we get a codimension-one expansive foliation with exceptional minimal set (and no compact leaves).
    
\end{example}

\section*{Aknowledgments}

The first author would like to thank Jana Rodríguez-Hertz and Raúl Ures for
their kind invitation to SUSTech University, where this research was conducted. He extends his gratitute to the Department of Mathematics of SUSTech, for their hospitality during this time.

\bibliographystyle{alpha}
\bibliography{expbiblio}

\end{document}